\documentclass[a4paper,12pt]{article}

\usepackage{amsmath,amssymb,amsthm}
\usepackage{latexsym}

\usepackage{graphics}
\usepackage{cite}
\usepackage{graphicx,psfrag}
\usepackage[bf, small]{titlesec}
\usepackage{authblk} 

\theoremstyle{plain}
\newtheorem{Thm}{Theorem}[section]
\newtheorem{Lem}[Thm]{Lemma}
\newtheorem{Prop}[Thm]{Proposition}
\newtheorem{Cor}[Thm]{Corollary}

\theoremstyle{definition}
\newtheorem*{Defi}{Definition}

\title{On the partial order competition dimensions of chordal graphs\footnotetext{\textit{Email addresses:}
\texttt{gaouls@snu.ac.kr} (Jihoon CHOI),
\texttt{srkim@snu.ac.kr} (Suh-Ryung KIM),
\texttt{jungyeunlee@gmail.com} (Jung Yeun LEE),
\texttt{sano@cs.tsukuba.ac.jp} (Yoshio SANO)}}

\author[1]{Jihoon CHOI
\thanks{This research was supported by Global Ph.D Fellowship Program through the National Research Foundation of Korea (NRF)
funded by the Ministry of Education (No.\ NRF-2015H1A2A1033541).}}
\author[1]{Suh-Ryung KIM
\thanks{This work was partially supported by
the National Research Foundation of Korea (NRF) grant funded
by the Korea government (MEST) (No.\ NRF-2015R1A2A2A01006885).}}
\author[1]{Jung Yeun LEE}
\author[2]{Yoshio SANO
\thanks{This work was supported by JSPS KAKENHI Grant Number 15K20885.}}
\affil[1]{Department of Mathematics Education,
Seoul National University, Seoul 151-742, Republic of Korea}
\affil[2]{Division of Information Engineering,
Faculty of Engineering, Information and Systems, \newline
University of Tsukuba, Ibaraki 305-8573, Japan}

\begin{document}

\maketitle

\begin{abstract}
Choi {\it et al.}
[{J.~Choi, K.~S.~Kim, S.~-R.~Kim, J.~Y.~Lee, and Y.~Sano}:
{On the competition graphs of $d$-partial orders},
\emph{Discrete Applied Mathematics} (2015),
\texttt{http://dx.doi.org/10.1016/j.dam.2015.11.004}]
introduced the notion of the partial order competition dimension of a graph.
It was shown that
complete graphs, interval graphs, and trees, which are chordal graphs, have
partial order competition dimensions at most three.

In this paper, we study the partial order competition dimensions of chordal graphs.
We show that
chordal graphs have partial order competition dimensions at most three
if the graphs are diamond-free.
Moreover, we also show the existence of
chordal graphs containing diamonds
whose partial order competition dimensions are greater than three.
\end{abstract}




\section{Introduction}

The \emph{competition graph} $C(D)$ of a digraph $D$ is an undirected graph which
has the same vertex set as $D$ and which has an edge $xy$
between two distinct vertices $x$ and $y$
if and only if for some vertex $z \in V$,
the arcs $(x,z)$ and $(y,z)$ are in $D$.

Let $d$ be a positive integer.
For $x = (x_1,x_2,\ldots, x_d)$,
$y = (y_1,y_2,\ldots, y_d) \in \mathbb{R}^d$,
we write
$x \prec y$
if $x_i<y_i$ for each $i=1, \ldots, d$.
For a finite subset $S$ of $\mathbb{R}^d$,
let $D_S$ be the digraph defined by $V(D_S) = S$ and
$A(D_S) = \{(x,v) \mid v, x \in S,
v \prec x \}$.
A digraph $D$ is called a \emph{$d$-partial order}
if there exists a finite subset $S$ of $\mathbb{R}^d$
such that $D$ is isomorphic to the digraph $D_S$.
A $2$-partial order is also called a \emph{doubly partial order}.
Cho and Kim~\cite{CK05} studied the competition graphs of doubly partial orders
and showed that interval graphs are exactly the
graphs having partial order competition dimensions at most two.
Several variants of competition graphs of doubly partial orders
also have been studied
(see \cite{KKR07, KLPPS09, KLPS14, LW09, PLK11, PS13, WL10}).

Choi {\it et al.} \cite{pocdim} introduced the notion of the partial order competition dimension of a graph.

\begin{Defi}
For a graph $G$,
the \emph{partial order competition dimension} of $G$, denoted by $\dim_{\text{{\rm poc}}}(G)$,
is the smallest nonnegative integer $d$
such that $G$ together with $k$ isolated vertices is the competition graph of
a $d$-partial order $D$ for some nonnegative integer $k$,
i.e.,
\[
\dim_{\text{{\rm poc}}}(G) :=
\min \{d \in \mathbb{Z}_{\geq 0} \mid
\exists k \in \mathbb{Z}_{\geq 0}, \exists S \subseteq \mathbb{R}^d
\text{ s.t. } G \cup I_k = C(D_S) \},
\]
where $\mathbb{Z}_{\geq 0}$ is the set of nonnegative integers
and $I_k$ is a set of $k$ isolated vertices.
\end{Defi}

Choi {\it et al.} \cite{pocdim} studied graphs
having small partial order competition dimensions,
and gave characterizations of graphs with partial order competition dimension $0$, $1$, or $2$
as follows.

\begin{Prop}\label{prop:dim0}
Let $G$ be a graph.
Then, $\dim_{\text{{\rm poc}}}(G)= 0$ if and only if $G = K_1$.
\end{Prop}

\begin{Prop}\label{prop:dim1}
Let $G$ be a graph.
Then, $\dim_{\text{{\rm poc}}}(G)= 1$ if and only if
$G = K_{t+1}$ or
$G = K_t \cup K_1$ for some positive integer $t$.
\end{Prop}

\begin{Prop}\label{prop:dim2}
Let $G$ be a graph.
Then, $\dim_{\text{{\rm poc}}}(G) = 2$
if and only if
$G$ is an interval graph which is neither $K_s$ nor $K_t \cup K_1$
for any positive integers $s$ and $t$.
\end{Prop}

\noindent
Choi {\it et al.} \cite{pocdim} also gave some families of graphs with partial order competition dimension three.

\begin{Prop}\label{prop:cycle}
If $G$ is a cycle of length at least four, then
$\dim_{\text{{\rm poc}}}(G) = 3$.
\end{Prop}

\begin{Thm}\label{thm:tree}
Let $T$ be a tree.
Then $\dim_{\text{{\rm poc}}}(T) \leq 3$,
and the equality holds if and only if $T$
is not a caterpillar.
\end{Thm}

In this paper, we study the partial order competition dimensions of chordal graphs.
We thought that most likely candidates for the family of graphs having partial order competition dimension
at most three are chordal graphs since both trees and interval graphs, which are chordal graphs,
have partial order competition dimensions at most three.
In fact,
we show that chordal graphs have partial order competition dimensions at most three
if the graphs are diamond-free.
However, contrary to our presumption,
we could show the existence of
chordal graphs with partial order competition dimensions greater than three.


\section{Preliminaries}

We say that two sets in $\mathbb{R}^d$ are \emph{homothetic}
if they are related by a geometric contraction or expansion.
Choi {\it et al.}~\cite{pocdim} gave a characterization of
the competition graphs of $d$-partial orders.
We state it in the case where $d=3$.

\begin{Thm} [\hskip-0.0025em \cite{pocdim}]\label{thm:intersectiongeneral}
A graph $G$ is the competition graph of a $3$-partial order
if and only if
there exists a family $\mathcal{F}$ of homothetic open equilateral   triangles
contained in the plane $\{x=(x_1,x_2,x_3) \in \mathbb{R}^3 \mid x_1+x_2+x_3 = 0 \}$ and
there exists a one-to-one correspondence $A: V(G) \to \mathcal{F}$
such that
\begin{itemize}
\item[{\rm ($\star$)}]
two vertices $v$ and $w$ are adjacent in $G$ if and only if
two elements $A(v)$ and $A(w)$ have the intersection containing the closure $\triangle(x)$
of an element $A(x)$ in $\mathcal{F}$.
\end{itemize}
\end{Thm}

\noindent
Choi {\it et al.}~\cite{pocdim} also gave a sufficient condition for
a graph being the competition graph of a $d$-partial order.
We state their result in the case where $d=3$.

\begin{Thm}[\hskip-0.0025em \cite{pocdim}]\label{thm:closed}
If $G$ is the intersection graph of
a finite family of homothetic closed equilateral triangles,
then $G$ together with sufficiently many new isolated vertices
is the competition graph of a $3$-partial order.
\end{Thm}

\noindent
By the definition of the partial order competition dimension of a graph,
we have the following:

\begin{Cor}\label{cor:closed}
If $G$ is the intersection graph of
a finite family of homothetic closed equilateral triangles,
then $\dim_{\text{{\rm poc}}}(G) \leq 3$.
\end{Cor}

\noindent
Note that the converse of Corollary \ref{cor:closed} is not true
by an example given by Choi {\em et al.}~\cite{pocdim}
(see Figure~\ref{counterexample}).
In this context, one can guess that it is not so easy to show that
a graph has partial order competition dimension greater than three.

\begin{figure}
\psfrag{a}{$v_1$}
\psfrag{b}{$v_2$}
\psfrag{c}{$v_3$}
\psfrag{d}{$v_4$}
\psfrag{e}{$v_5$}
\psfrag{f}{$v_6$}
\psfrag{g}{$v_7$}
\psfrag{A}{$A^2(\mathbf{v}_1)$}
\psfrag{B}{$A^2(\mathbf{v}_2)$}
\psfrag{C}{$A^2(\mathbf{v}_3)$}
\psfrag{D}{$A^2(\mathbf{v}_4)$}
\psfrag{E}{$A^2(\mathbf{v}_5)$}
\psfrag{F}{$A^2(\mathbf{v}_6)$}
\psfrag{G}{$A^2(\mathbf{v}_7)$}
\psfrag{h}{$G$}
\begin{center}
\includegraphics[height=5cm]{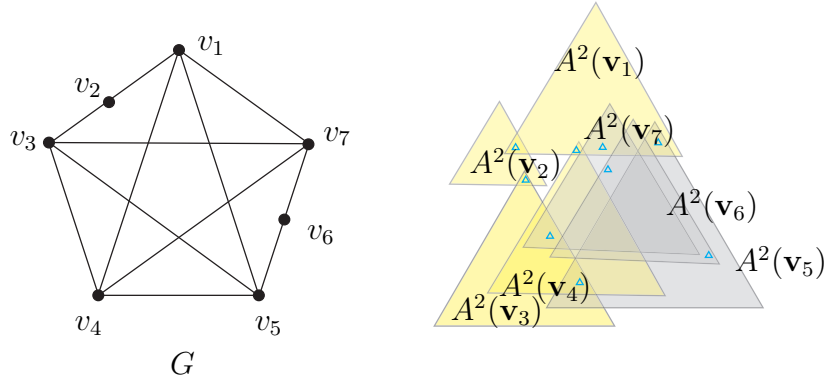}
\end{center}
\caption{A subdivision $G$ of $K_5$ and a family of homothetic equilateral triangles making
$G$ together with $9$ isolated vertices into
the competition graph of a $3$-partial order, which is given in \cite{pocdim}}
\label{counterexample}
\end{figure}

The correspondence $A$ in Theorem~\ref{thm:intersectiongeneral} can be precisely described as follows:
Let $\mathcal{H}:=\{x = (x_1, x_2, x_3) \in \mathbb{R}^3
\mid x_1 + x_2 + x_3 = 0 \}$
and $\mathcal{H}_+ := \{x = (x_1, x_2, x_3) \in \mathbb{R}^3 \mid x_1 + x_2 + x_3 > 0 \}$.
For a point $v=(v_1, v_2, v_3) \in \mathcal{H}_+$,
let $p_{1}^{(v)}$, $p_{2}^{(v)}$, and $p_{3}^{(v)}$
be points in $\mathbb{R}^3$ defined by
$p_{1}^{(v)} := (-v_2-v_3, v_2, v_3)$,
$p_{2}^{(v)} := (v_1, -v_1-v_3, v_3)$, and
$p_{3}^{(v)} := (v_1, v_2, -v_1-v_2)$,
and let $\triangle(v)$ be the convex hull of the points
$p_{1}^{(v)}$, $p_{2}^{(v)}$, and $p_{3}^{(v)}$, i.e.,
$
\triangle(v) := \text{{\rm Conv}}(p_{1}^{(v)},p_{2}^{(v)},p_{3}^{(v)})
= \left\{ \sum_{i=1}^3 \lambda_i p_{i}^{(v)}
\mid \sum_{i=1}^3 \lambda_i=1, \lambda_i \geq 0 \ (i=1,2,3) \right\}.
$
Then it is easy to check that $\triangle(v)$ is an closed equilateral triangle
which is contained in the plane $\mathcal{H}$.
Let $A(v)$ be the relative interior of the closed triangle $\triangle(v)$, i.e.,
$
A(v) := \text{rel.int}(\triangle(v))
= \left\{ \sum_{i=1}^3 \lambda_i p_{i}^{(v)}
\mid \sum_{i=1}^3 \lambda_i=1, \lambda_i > 0 \ (i=1,2,3) \right\}.
$
Then
$A(v)$ and $A(w)$ are homothetic for any $v,w \in \mathcal{H}_+$.

For $v \in \mathcal{H}_+$ and $(i,j) \in \{(1,2),(2,3),(1,3)\}$,
let $l_{ij}^{(v)}$ denote the line through
the two points $p^{(v)}_{i}$ and $p^{(v)}_{j}$, i.e.,
$
l_{ij}^{(v)} := \{ x \in \mathbb{R}^3 \mid
x = \alpha p^{(v)}_{i} + (1 - \alpha) p^{(v)}_{j}, \alpha \in \mathbb{R} \},
$
and let $R_{ij}(v)$ denote the following region:
\[
R_{ij}(v) := \{ x \in \mathbb{R}^3 \mid
x = (1-\alpha - \beta)p^{(v)}_{k} + \alpha p^{(v)}_{i} + \beta p^{(v)}_{j} ,
0 \leq \alpha \in \mathbb{R}, 0 \leq \beta \in \mathbb{R}, \alpha + \beta \geq 1 \},
\]
where $k$ is the element in $\{1,2,3\} \setminus \{i,j\}$;
for $k \in \{1,2,3\}$, let $R_{k}(v)$ denote the following region:
\[
R_{k}(v) := \{ x \in \mathbb{R}^3 \mid
x = (1 + \alpha + \beta)p^{(v)}_{k} - \alpha p^{(v)}_{i} - \beta p^{(v)}_{j},
0 \leq \alpha \in \mathbb{R}, 0 \leq \beta \in \mathbb{R} \},
\]
where $i$ and $j$ are elements such that
$\{i,j,k\} = \{1,2,3\}$.
(See Figure~\ref{region} for an illustration.)

\begin{figure}
\psfrag{A}{ $p^{(v)}_{1}$}
\psfrag{B}{ $p^{(v)}_{2}$}
\psfrag{C}{ $p^{(v)}_{3}$}
\psfrag{D}{ $\triangle(v)$}
\psfrag{E}{ $R_{23}(v)$}
\psfrag{F}{ $R_3(v)$}
\psfrag{G}{ $R_{13}(v)$}
\psfrag{H}{ $R_1(v)$}
\psfrag{I}{ $R_{12}(v)$}
\psfrag{J}{ $R_2(v)$}
\psfrag{K}{ $l_{12}^{(v)}$}
\psfrag{L}{ $l_{13}^{(v)}$}
\psfrag{M}{ $l_{23}^{(v)}$}
\begin{center}
\includegraphics[height=5cm]{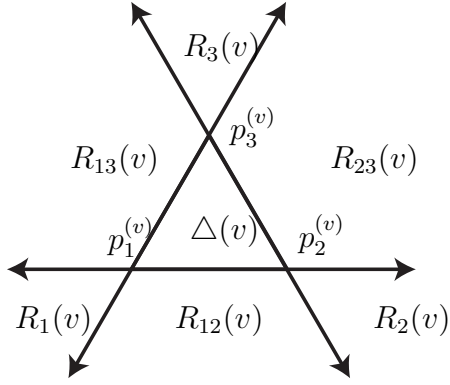}
\end{center}
\vskip-1em
\caption{The regions determined by $v$. By our assumption, for any vertex $u$ of a graph considered in this paper, $p_1^{(u)}$, $p_2^{(u)}$, $p_3^{(u)}$ correspond to $p_1^{(v)}$, $p_2^{(v)}$, $p_3^{(v)}$ respectively.
}
\label{region}
\end{figure}

If a graph $G$ satisfies $\dim_{\text{{\rm poc}}}(G) \leq 3$, then, by Theroem~\ref{thm:intersectiongeneral}, we may assume that $V(G) \subseteq \mathcal{H}_+$ by translating each of the vertices of $G$ in the same direction and by the same amount.

\begin{Lem}\label{lem:not-incl1}
Let $D$ be a $3$-partial order and
let $G$ be the competition graph of $D$.
Suppose that $G$ contains an induced path $uvw$ of length two.
Then neither $A(u) \cap A(v) \subseteq A(w)$ nor $A(v) \cap A(w) \subseteq A(u)$.
\end{Lem}

\begin{proof}
We show by contradiction.
Suppose that $A(u) \cap A(v) \subseteq A(w)$ or $A(v) \cap A(w) \subseteq A(u)$.
By symmetry, we may assume without loss of generality that $A(u) \cap A(v) \subseteq A(w)$.
Since $u$ and $v$ are adjacent in $G$,
there exists a vertex $a \in V(G)$ such that
$\triangle(a) \subseteq A(u) \cap A(v)$
by Theorem~\ref{thm:intersectiongeneral}.
Therefore
$\triangle(a) \subseteq A(w)$.
Since $\triangle(a) \subseteq A(u)$,
$u$ and $w$ are adjacent in $G$ by Theorem~\ref{thm:intersectiongeneral},
which is a contradiction to the assumption that
$u$ and $w$ are not adjacent in $G$.
Hence the lemma holds.
\end{proof}

\begin{Defi}
For $v,w \in \mathcal{H}_+$,
we say that $v$ and $w$ are
\emph{crossing} if
$A(v) \cap A(w) \neq \emptyset$,
$A(v) \setminus A(w) \neq \emptyset$, and
$A(w) \setminus A(v) \neq \emptyset$.
\end{Defi}

\begin{Lem}\label{lem:not-incl2}
Let $D$ be a $3$-partial order and
let $G$ be the competition graph of $D$.
Suppose that $G$ contains an induced path $xuvw$ of length three.
Then $u$ and $v$ are crossing.
\end{Lem}

\begin{proof}
Since $u$ and $v$ are adjacent in $G$, there exists a vertex $a \in V(G)$
such that $\triangle(a) \subseteq A(u) \cap A(v)$
by Theorem~\ref{thm:intersectiongeneral}.
Therefore $A(u) \cap A(v) \neq \emptyset$.
If $A(v) \subseteq A(u)$, then $A(v) \cap A(w) \subseteq A(u)$,
which contradicts Lemma~\ref{lem:not-incl1}.
Thus $A(v) \setminus A(u) \neq \emptyset$.
If $A(u) \subseteq A(v)$, then $A(x) \cap A(u) \subseteq A(v)$,
which contradicts Lemma~\ref{lem:not-incl1}.
Thus $A(u) \setminus A(v) \neq \emptyset$.
Hence $u$ and $v$ are crossing.
\end{proof}

\begin{Lem}\label{lem:intersecting3}
If $v$ and $w$ in $\mathcal{H}_+$ are crossing,
then $p_k^{(x)} \in \triangle(y)$ for some $k \in \{1,2,3\}$
where $\{x,y\} = \{v,w\}$.
\end{Lem}

\begin{proof}
Since $v$ and $w$ are crossing,
we have
$A(v) \cap A(w) \neq \emptyset$,
$A(v) \setminus A(w) \neq \emptyset$, and
$A(w) \setminus A(v) \neq \emptyset$.
Then one of the vertices of the triangles $\triangle(v)$ and $\triangle(w)$
is contained in the other triangle, thus the lemma holds.
\end{proof}

\begin{Defi}
For $k \in \{1,2,3\}$,
we define a binary relation $\stackrel{k}{\rightarrow}$ on $\mathcal{H}_+$ by
\[
x \stackrel{k}{\rightarrow} y
\quad \Leftrightarrow \quad
\text{ $x$ and $y$ are crossing, and }  p_k^{(y)} \in \triangle(x)
\]
for any $x, y \in \mathcal{H}_+$.
\end{Defi}

\begin{Lem}\label{lem:transitive}
Let $x,y,z \in \mathcal{H}_+$.
Suppose that $x \stackrel{k}{\rightarrow} y$
and $y \stackrel{k}{\rightarrow} z$ for some $k \in \{1,2,3\}$
and that $x$ and $z$ are crossing.
Then $x \stackrel{k}{\rightarrow} z$.
\end{Lem}

\begin{proof}
Since $x \stackrel{k}{\rightarrow} y$, $p_l^{(x)} \not\in R_i(y) \cup R_{ij}(y) \cup R_j(y)$ for each $l \in \{1,2,3\}$,
where $\{i,j,k\} = \{1,2,3\}$
Since $y \stackrel{k}{\rightarrow} z$, $p_l^{(z)} \in R_i(y) \cup R_{ij}(y) \cup R_j(y)$ for each $l \in \{i,j\}$.
Since $x$ and $z$ are crossing, $p_k^{(z)} \in \triangle(x)$.
\end{proof}

\begin{Defi}
For $k \in \{1,2,3\}$,
a sequence $(v_1, \ldots, v_m)$ of $m$ points in $\mathcal{H}_+$, where $m \geq 2$,
is said to be \emph{consecutively tail-biting in Type $k$}
if $v_i \stackrel{k}{\rightarrow} v_j$ for any $i < j$
(see Figure~\ref{consecutive}).
A finite set $V$ of points in $\mathcal{H}_+$
is said to be \emph{consecutively tail-biting}
if there is an ordering $(v_1, \ldots, v_m)$ of $V$
such that $(v_1, \ldots, v_m)$ is consecutively tail-biting.
\end{Defi}

\begin{figure}
\psfrag{A}{ $A(v_1)$}
\psfrag{B}{ $A(v_2)$}
\psfrag{C}{ $A(v_3)$}
\psfrag{D}{ $A(v_4)$}
\psfrag{E}{(a)}
\psfrag{F}{(b)}
\psfrag{G}{(c)}
\begin{center}
\includegraphics[height=5cm]{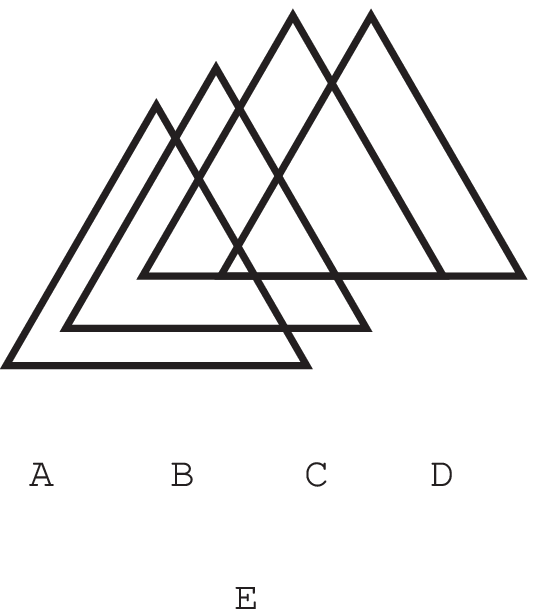} \hskip2.5em
\includegraphics[height=5cm]{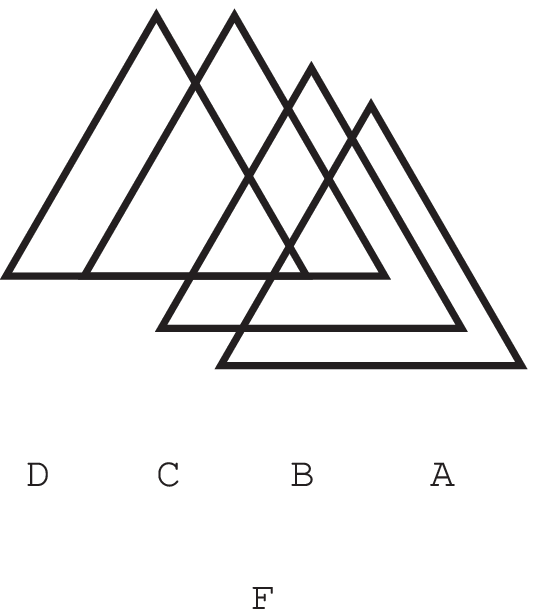} \hskip2.5em
\includegraphics[height=5cm]{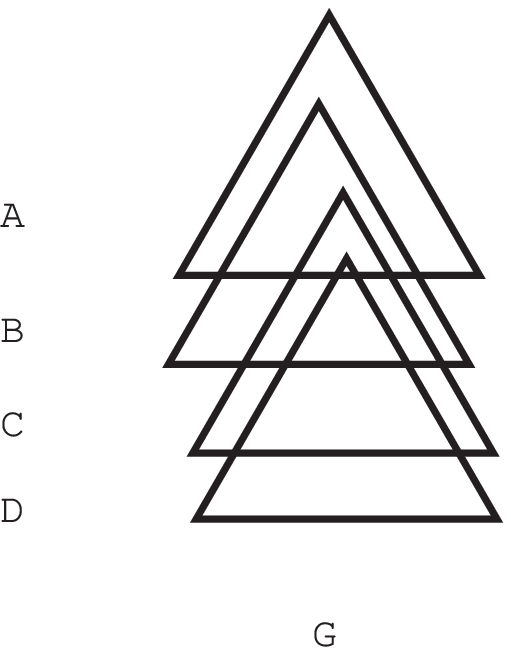}
\end{center}
\caption{The sequences $(v_1,v_2,v_3,v_4)$ in (a), (b), (c) are consecutively tail-biting of Type 1, 2, 3, respectively.}
\label{consecutive}
\end{figure}

\section{The partial order competition dimensions of diamond-free chordal graphs}\label{sec:chordal}

In this section, we show that
a chordal graph has partial order competition dimension at most three
if it is diamond-free.

A \emph{block graph} is a graph such that each of its maximal $2$-connected subgraphs
is a complete graph.
The following is well-known.

\begin{Lem}[\hskip-0.0025em {\cite[Proposition 1]{BM86}}]\label{lem:blockchara}
A graph is a block graph
if and only if the graph is a diamond-free chordal graph.
\end{Lem}

Note that
a block graph having no cut vertex is a disjoint union of complete graphs.
For block graphs having cut vertices, the following lemma holds.

\begin{Lem}\label{lem:blockgraph}
Let $G$ be a block graph having at least one cut vertex.
Then $G$ has a maximal clique that contains exactly one cut vertex.
\end{Lem}

\begin{proof}
Let $H$ be the subgraph induced by the cut vertices of $G$.
By definition, $H$ is obviously a block graph, so $H$ is chordal and there is a simplicial vertex $v$ in $H$.
Since $v$ is a cut vertex of $G$, $v$ belongs to at least two maximal cliques of $G$.
Suppose that each maximal clique containing $v$ contains another cut vertex of $G$.
Take two maximal cliques $X_1$ and $X_2$ of $G$ containing $v$
and let $x$ and $y$ be cut vertices of $G$ belonging to $X_1$ and $X_2$, respectively.
Then both $x$ and $y$ are adjacent to $v$ in $H$.
Since $G$ is a block graph,
$X_1 \setminus \{v\}$ and $X_2 \setminus \{v\}$
are contained in distinct connected components of $G-v$.
This implies that $x$ and $y$ are not adjacent in $H$,
which contradicts the choice of $v$.
Therefore there is a maximal clique $X$ containing $v$
without any other cut vertex of $G$.
\end{proof}


\begin{Lem}\label{lem:blockgraph2}
Every block graph $G$ is the intersection graph
of a family $\mathcal{F}$ of homothetic closed equilateral triangles
in which every clique of $G$ is consecutively tail-biting.
\end{Lem}

\begin{proof}
We show by induction on the number of cut vertices of $G$.
If a block graph has no cut vertex, then it is a disjoint union of complete graphs and the statement
is trivially true as the vertices of each complete subgraph can be formed as a sequence which is consecutively tail-biting (refer to Figure~\ref{consecutive}).

Assume that the statement is true for any block graph $G$ with $m$ cut vertices where $m \geq 0$.
Now we take a block graph $G$ with $m+1$ cut vertices.
By Lemma~\ref{lem:blockgraph},
there is a maximal clique $X$ that contains exactly one cut vertex, say $w$.
By definition, the vertices of $X$ other than $w$ are simplicial vertices.

Deleting the vertices of $X$ other than $w$
and the edges adjacent to them,
we obtain a block graph $G^*$ with $m$ cut vertices.
Then, by the induction hypothesis, $G^*$ is the intersection graph of a family ${\cal F}^*$
of homothetic closed equilateral triangles satisfying the statement.
We consider the triangles corresponding to $w$.
Let $C$ and $C'$ be two maximal cliques of $G^*$ containing $w$.
By the induction hypothesis,
the vertices of $C$ and $C'$ can be ordered as
$v_{1}, v_{2}, \ldots, v_{l}$ and $v'_{1}, v'_{2}, \ldots, v'_{l'}$, respectively,
so that $v_{i} \stackrel{k}{\rightarrow} v_{j}$ if $i < j$, for some $k \in \{1,2,3\}$ and
that $v'_{i'} \stackrel{k'}{\rightarrow} v'_{j'}$ if $i' < j'$, for some $k' \in \{1,2,3\}$.

Suppose that $\triangle(v_i) \cap \triangle(v'_j) \neq \emptyset$ for $v_i$ and $v'_j$ which are distinct from $w$.
Then $v_i$ and $v'_j$ are adjacent in $G^*$, which implies the existence of a diamond in $G$
since maximal cliques have size at least two.
We have reached a contradiction to Lemma~\ref{lem:blockchara} and so $\triangle(v_i) \cap \triangle(v'_j) = \emptyset$ for any $i,j$.
Therefore there is a segment of a side on $\triangle(w)$ (with a positive length) that does not intersect with the triangle assigned to any vertex in $G^*$ other than $w$
since there are finitely many maximal cliques in $G^*$ that contain $w$.
If the side belongs to $l_{ij}^{(w)}$ for $i,j \in \{1,2,3\}$,
then we may order the deleted vertices and assign the homothetic closed equilateral triangles
with sufficiently small sizes to them
so that the closed neighborhood of $v$ is consecutively tail-biting in Type $k$ for $k \in \{1,2,3\} \setminus \{i,j\}$
and none of the triangles intersects with the triangle corresponding to any vertex other than $w$ in $G^*$.
It is not difficult to see that the set of the triangles in $\mathcal{F}^*$
together with the triangles just obtained is the one desired for $\mathcal{F}$.
\end{proof}

\begin{Thm}
For any diamond-free chordal graph $G$, $\dim_{\text{{\rm poc}}}(G) \leq 3$.
\end{Thm}

\begin{proof}
The theorem follows from
Corollary~\ref{cor:closed}
and Lemma~\ref{lem:blockgraph2}.
\end{proof}

\section{Chordal graphs having partial order competition dimension greater than three}\label{sec:dimpocmorethanthree}

In this section,
we present infinitely many chordal graphs $G$
with $\dim_{\text{{\rm poc}}}(G) > 3$.
We first show two lemmas which will be repeatedly used
in the proof of the theorem in this section.

\begin{Lem}\label{lem:3triangles}
Let $D$ be a $3$-partial order and
let $G$ be the competition graph of $D$.
Suppose that $G$
contains a diamond $K_4-e$ as an induced subgraph,
where $u$, $v$, $w$, $x$ are the vertices of the diamond and $e=vx$.
If the sequence $(u, v, w)$ is consecutively tail-biting in Type $k$ for some $k \in \{1,2,3\}$,
then $p_i^{(x)} \in R_i(v)$ and $p_j^{(x)} \notin R_j(v)$ hold or $p_i^{(x)} \notin R_i(v)$ and $p_j^{(x)} \in R_j(v)$ hold where $\{i,j,k\} = \{1,2,3\}$.
\end{Lem}

\begin{proof}
Without loss of generality, we may assume that $k=3$.
We first claim that $p_1^{(x)} \in R_1(v) \cup R_2(v) \cup R_{12}(v)$.
Suppose not.
Then $p_1^{(x)} \in R:=\mathcal{H} \setminus (R_1(v) \cup R_2(v) \cup R_{12}(v))$.
Since $A(x)$ and $A(v)$ are homothetic, $A(x) \subseteq R$.
Thus $A(w) \cap A(x) \subseteq A(w) \cap R$.
Since $(u,v,w)$ is consecutively tail-biting in Type 3, $A(w) \cap R \subseteq A(v)$.
Therefore $A(w) \cap A(x)\subseteq A(v)$, which contradicts Lemma~\ref{lem:not-incl1}.
Thus $p_1^{(x)} \in R_1(v) \cup R_2(v) \cup R_{12}(v)$.
By symmetry, $p_2^{(x)} \in R_1(v) \cup R_2(v) \cup R_{12}(v)$.

Suppose that both $p_1^{(x)}$ and $p_2^{(x)}$ are in $R_{12}(v)$.
Since $A(x)$ and $A(v)$ are homothetic, $A(x) \cap R \subseteq A(v)$.
By the hypothesis that $(u,v,w)$ is consecutively tail-biting in Type 3,
we have $A(u) \subseteq R$. Therefore $A(x) \cap A(u) \subseteq A(x) \cap R$.
Thus $A(x) \cap A(u) \subseteq A(v)$, which contradicts Lemma~\ref{lem:not-incl1}.
Therefore $p_1^{(x)} \in R_1(v) \cup R_2(v)$ or $p_2^{(x)} \in R_1(v) \cup R_2(v)$.
Since $p_1^{(x)} \in R_2(v)$ (resp. $p_2^{(x)} \in R_1(v)$) implies
$p_2^{(x)} \in R_2(v)$ (resp. $p_1^{(x)} \in R_1(v)$), which is impossible, we have
$p_1^{(x)} \in R_1(v)$ or $p_2^{(x)} \in R_2(v)$.

Suppose that both $p_1^{(x)} \in R_1(v)$ and $p_2^{(x)} \in R_2(v)$ hold.
Then $A(v) \subseteq A(x)$ since $A(v)$ and $A(x)$ are homothetic.
Then $A(u) \cap A(v) \subseteq A(x)$, which contradicts Lemma~\ref{lem:not-incl1}.
Hence $p_1^{(x)} \in R_1(v)$ and $p_2^{(x)} \notin R_2(v)$ hold or $p_1^{(x)} \notin R_1(v)$ and $p_2^{(x)} \in R_2(v)$ hold.
\end{proof}

\begin{figure}
\psfrag{t}{\small $t$}
\psfrag{u}{\small $u$}
\psfrag{v}{\small $v$}
\psfrag{w}{\small $w$}
\psfrag{x}{\small $x$}
\psfrag{y}{\small $y$}
\begin{center}
\includegraphics{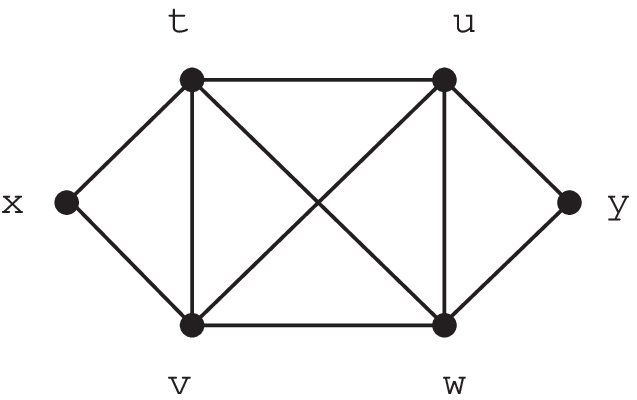}
\end{center}
\caption{The graph $\overline{\mathrm{H}}$}
\label{fig:config}
\end{figure}

Let $\overline{\mathrm{H}}$ be the graph on vertex set $\{t,u,v,w,x,y\}$ such that $\{t,u,v,w\}$ forms a complete graph $K_4$, $x$ is adjacent to only $t$ and $v$, and $y$ is adjacent to only $u$ and $w$ in $\overline{\mathrm{H}}$
(see Figure~\ref{fig:config} for an illustration).

\begin{Lem}\label{lem:4triangles}
Let $D$ be a $3$-partial order and let $G$ be the competition graph of $D$.
Suppose that $G$ contains the graph $\overline{\mathrm{H}}$ as an induced subgraph and
$(t, u, v, w)$ is consecutively tail-biting in Type $k$ for some $k \in \{1,2,3\}$.
Then, for $i,j$ with $\{i,j,k\} = \{1,2,3\}$,
$p_i^{(x)} \in R_i(u)$ implies $p_j^{(y)} \in R_j(v)$.
\end{Lem}

\begin{proof}
Without loss of generality, we may assume that $k=3$.
It is sufficient to show that $p_1^{(x)} \in R_1(u)$ implies $p_2^{(y)} \in R_2(v)$.
Now suppose that $p_1^{(x)} \in R_1(u)$.
Since $(t,u,v,w)$ is a tail-biting sequence of Type 3,
$(t,u,v)$ and $(u,v,w)$ are tail-biting sequences of Type 3.
Since $\{t,u,v,x\}$ induces a diamond and
$(t,u,v)$ is a consecutively tail-biting sequence of Type 3,
it follows from Lemma~\ref{lem:3triangles} that
$p_1^{(x)} \in R_1(u)$ and $p_2^{(x)} \not\in R_2(u)$ hold or $p_1^{(x)} \notin R_1(u)$ and $p_2^{(x)} \in R_2(u)$ hold.
Since $p_1^{(x)} \in R_1(u)$, it must hold that $p_1^{(x)} \in R_1(u)$ and $p_2^{(x)} \not\in R_2(u)$.
Since $A(u)$ and $A(x)$ are homothetic and $p_1^{(x)} \in R_1(u)$,
we have $A(u) \subseteq A(x) \cup R_{23}(x)$.

Since $\{u,v,w,y\}$ induces a diamond and
$(u,v,w)$ is a consecutively tail-biting sequence of Type 3, it follows from Lemma~\ref{lem:3triangles} that  $p_1^{(y)} \in R_1(v)$ and $p_2^{(y)} \not\in R_2(v)$ hold or $p_1^{(y)} \not\in R_1(v)$ and $p_2^{(y)} \in R_2(v)$ hold.
We will claim that the latter is true as it implies $p_2^{(y)} \in R_2(v)$.
To reach a contradiction, suppose the former, that is,  $p_1^{(y)} \in R_1(v)$ and $p_2^{(y)} \not\in R_2(v)$.
Since $A(v)$ and $A(y)$ are homothetic and $p_1^{(y)} \in R_1(v)$, we have $A(v) \subseteq A(y) \cup R_{23}(y)$.
We now show that $A(x) \cap A(v) \subseteq A(y)$.
Take any $a \in A(x) \cap A(v)$.
Since $A(v) \subseteq A(y) \cup R_{23}(y)$, we have $a \in A(y) \cup R_{23}(y)$.
Suppose that $a \not\in A(y)$. Then $a \in R_{23}(y)$.
This together with the fact that $a \in A(x)$ 
implies $A(y) \cap R_{23}(x) = \emptyset$.
Since $A(u) \subseteq A(x) \cup R_{23}(x)$, we have
\begin{align*}
A(u) \cap A(y)
&\subseteq (A(x) \cup R_{23}(x)) \cap A(y) \\
&= (A(x) \cap A(y)) \cup (R_{23}(x)) \cap A(y)) \\
&= (A(x) \cap A(y)) \cup \emptyset \\
&= A(x) \cap A(y) \subseteq A(x).
\end{align*}
Therefore $A(u) \cap A(y) \subseteq A(u) \cap A(x)$.
Since $u$ and $y$ are adjacent in $G$,
there exists $b \in V(G)$ such that
$\triangle(b) \subseteq A(u) \cap A(y)$.
Then $\triangle(b) \subseteq A(u) \cap A(x)$,
which is a contradiction to the fact that $u$ and $x$ are not adjacent in $G$.
Thus $a \notin R_{23}(y)$ and so $a \in A(y)$.
Hence we have shown that $A(x) \cap A(v) \subseteq A(y)$.
Since $x$ and $v$ are adjacent in $G$,
there exists $c \in V(G)$ such that
$\triangle(c) \subseteq A(x) \cap A(v)$.
Then $\triangle(c) \subseteq A(v) \cap A(y)$,
which is a contradiction to the fact that $v$ and $y$ are not adjacent in $G$.
Thus we have $p_1^{(y)} \not\in R_1(v)$ and $p_2^{(y)} \in R_2(v)$.
Hence the lemma holds.
\end{proof}

\begin{Defi}\label{def:expansion}
For a positive integer $n$,
let $G_n$ be the graph obtained from the complete graph $K_n$
by adding a path of length $2$
for each pair of vertices of $K_n$,
i.e.,
$V(G_n) = \{ v_i \mid 1 \leq i \leq n \}
\cup \{ v_{ij} \mid 1 \leq i < j \leq n \}$
and
$E(G_n) = \{ v_i v_j \mid 1 \leq i < j \leq n \}
\cup \{ v_i v_{ij} \mid 1 \leq i < j \leq n \}
\cup \{ v_j v_{ij} \mid 1 \leq i < j \leq n \}$.
\end{Defi}

\begin{Defi}
For a positive integer $m$,
the \emph{Ramsey number} $r(m,m,m)$
is the smallest positive integer $r$
such that any $3$-edge-colored complete graph $K_r$ of order $r$
contains a monochromatic complete graph $K_m$ of order $m$.
\end{Defi}


\begin{Lem}\label{lem:tail}
Let $m$ be a positive integer at least $3$ and
let $n$ be an integer greater than or equal to
the Ramsey number $r(m,m,m)$.
If $\dim_{{\rm poc}}(G_n) \leq 3$,
then
there exists a sequence
$(x_1, \ldots, x_m)$ of vertices of $G_n$
such that $\{ x_1, \ldots, x_m \}$
is a clique of $G_n$
and that any subsequence
$(x_{i_1}, \ldots, x_{i_l})$ of $(x_1, \ldots, x_m)$
is consecutively tail-biting,
where $2 \leq l \leq m$
and $1 \leq i_1 < \cdots < i_l \leq m$.
\end{Lem}

\begin{proof}
Since the vertices $v_i$ and $v_j$ of $G_n$
are internal vertices of an induced path
of length three by the definition of $G_n$,
it follows from Lemma \ref{lem:not-incl2} that
the vertices $v_i$ and $v_j$ of $G_n$
are crossing.
By Lemma \ref{lem:intersecting3}, for any $1 \leq i < j \leq n$,
there exists $k \in \{1,2,3\}$ such that
$v_i \stackrel{k}{\rightarrow} v_j$ or $v_j \stackrel{k}{\rightarrow} v_i$.
Now we define an edge-coloring
$c:\{v_iv_j \mid 1 \leq i<j \leq n\} \to \{1,2,3\}$
as follows:
For $1 \leq i < j \leq n$,
we let $c(v_iv_j) = k$ so that
$v_i \stackrel{k}{\rightarrow} v_j$ or $v_j \stackrel{k}{\rightarrow} v_i$.
Then, by the definition of $r(m,m,m)$,
$K_n$ contains a monochromatic complete subgraph $K$ with $m$ vertices.

Suppose that the edges of $K$ have color $k$, where $k \in \{1,2,3\}$.
We assign an orientation to each edge $xy$ of $K$ so that $x$ goes toward $y$
if $x \stackrel{k}{\rightarrow} y$.
In that way, we obtain a tournament $\overrightarrow{K}$ with $m$ vertices.
It is well-known that every tournament has a directed Hamiltonian path.
Therefore, $\overrightarrow{K}$ has a directed Hamiltonian path.
Let $x_1 \rightarrow x_2 \rightarrow \cdots \rightarrow x_m$ be
a directed Hamiltonian path of $\overrightarrow{K}$.
Then, by Lemma~\ref{lem:transitive}, $x_i \stackrel{k}{\rightarrow} x_j$ for any $i < j$.
Thus
any subsequence
$(x_{i_1}, \ldots, x_{i_l})$ of $(x_1, \ldots, x_m)$
is consecutively tail-biting,
where $2 \leq l \leq m$
and $1 \leq i_1 < \cdots < i_l \leq m$.
\end{proof}

Since the graph $G_n$ is chordal,
the following theorem shows the existence of chordal graphs
with partial order competition dimensions greater than three.
Given a graph $G$ and a set $X$ consisting of six vertices in $G$,
we say that
\emph{$X$ induces an $\overline{\mathrm{H}}$} if it induces a subgraph of $G$ isomorphic to $\overline{\mathrm{H}}$.

\begin{Thm}
For $n \geq r(5,5,5)$, $\dim_{{\rm poc}}(G_n) > 3$.
\end{Thm}

\begin{proof}
We prove by contradiction.
Suppose that $\dim_{{\rm poc}}(G_n) \leq 3$
for some $n \geq r(5,5,5)$.
By Lemma~\ref{lem:tail},
$G_n$ contains
a consecutively tail-biting sequence $(v_1, \ldots, v_5)$ of five vertices in Type $k$
such that $\{v_1, \ldots, v_5\}$ is a clique of $G_n$
and that
$(v_{i_1}, v_{i_2}, v_{i_3})$
is a consecutively tail-biting sequence for any $1 \leq i_1 < i_2 < i_3 \leq 5$
and
$(v_{i_1}, v_{i_2}, v_{i_3}, v_{i_4})$
is a consecutively tail-biting sequence for any $1 \leq i_1 < i_2 < i_3 < i_4 \leq 5$.
Without loss of generality, we may assume that $k=3$.

Since $\{v_1,v_2,v_3,v_{13}\}$ induces a diamond and
$(v_1,v_2,v_3)$ is a consecutively tail-biting sequence of Type 3,
it follows from Lemma~\ref{lem:3triangles} that
$p_1^{(v_{13})} \in R_1(v_2)$ and $p_2^{(v_{13})} \not\in R_2(v_2)$ hold or $p_1^{(v_{13})} \notin R_1(v_2)$ and $p_2^{(v_{13})} \in R_2(v_2)$ hold.

We first suppose that $p_1^{(v_{13})} \in R_1(v_2)$ and $p_2^{(v_{13})} \not\in R_2(v_2)$.
Since $\{v_1,v_2,v_3,v_4,v_{13},v_{24}\}$ induces an $\overline{\mathrm{H}}$ and
$(v_1,v_2,v_3,v_4)$ is a consecutively tail-biting sequence of Type 3,
it follows from Lemma~\ref{lem:4triangles} and $p_1^{(v_{13})} \in R_1(v_2)$ that
$p_2^{(v_{24})} \in R_2(v_3)$.
Since $\{v_1,v_2,v_3,v_5,v_{13},v_{25}\}$ induces an $\overline{\mathrm{H}}$ and
$(v_1,v_2,v_3,v_5)$ is a consecutively tail-biting sequence of Type 3,
it follows from Lemma~\ref{lem:4triangles} and $p_1^{(v_{13})} \in R_1(v_2)$ that
\begin{equation}\label{eqn:three}
p_2^{(v_{25})} \in R_2(v_3).
\end{equation}
Since $\{v_2,v_3,v_4,v_5,v_{24},v_{35}\}$ induces an $\overline{\mathrm{H}}$ and
$(v_2,v_3,v_4,v_5)$ is a consecutively tail-biting sequence of Type 3,
it follows from Lemma~\ref{lem:4triangles} and $p_2^{(v_{24})} \in R_2(v_3)$ that
\begin{equation}\label{eqn:four}
p_1^{(v_{35})} \in R_1(v_4).
\end{equation}
Since $\{v_1,v_3,v_4,v_{14}\}$ induces a diamond and
$(v_1,v_3,v_4)$ is a consecutively tail-biting sequence of Type 3,
it follows from Lemma~\ref{lem:3triangles} that
$p_1^{(v_{14})} \in R_1(v_3)$ and $p_2^{(v_{14})} \not\in R_2(v_3)$ hold or $p_1^{(v_{14})} \notin R_1(v_3)$ and $p_2^{(v_{14})} \in R_2(v_3)$ hold.
Suppose that $p_1^{(v_{14})} \in R_1(v_3)$ and $p_2^{(v_{14})} \not\in R_2(v_3)$.
Since $\{v_1,v_3,v_4,v_5,v_{14},v_{35}\}$ induces an $\overline{\mathrm{H}}$ and
$(v_1,v_3,v_4,v_5)$ is a consecutively tail-biting sequence of Type 3,
it follows from Lemma~\ref{lem:4triangles} and $p_1^{(v_{14})} \in R_1(v_3)$ that
\begin{equation}\label{eqn:four-2}
p_2^{(v_{35})} \in R_2(v_4).
\end{equation}
Since $\{v_3,v_4,v_5,v_{35}\}$ induces a diamond and
$(v_3,v_4,v_5)$ is a consecutively tail-biting sequence of Type 3,
it follows from Lemma~\ref{lem:3triangles} that
$p_1^{(v_{35})} \in R_1(v_4)$ and $p_2^{(v_{35})} \notin R_2(v_4)$ hold or  $p_1^{(v_{35})} \notin R_1(v_4)$ and $p_2^{(v_{35})} \in R_2(v_4)$ hold,
which is a contradiction to the fact that both (\ref{eqn:four}) and (\ref{eqn:four-2}) hold.
Thus 
\begin{equation}\label{eqn:five}
p_1^{(v_{14})} \not\in R_1(v_3) \text{ and } p_2^{(v_{14})} \in R_2(v_3).
\end{equation}
Since $\{v_1,v_2,v_4,v_{14}\}$ induces a diamond and
$(v_1,v_2,v_4)$ is a consecutively tail-biting sequence of Type 3,
it follows from Lemma~\ref{lem:3triangles} that
$p_1^{(v_{14})} \in R_1(v_2)$ and $p_2^{(v_{14})} \notin R_2(v_2)$ hold or $p_1^{(v_{14})} \notin R_1(v_2)$ and $p_2^{(v_{14})} \in R_2(v_2)$ hold.
Suppose that $p_1^{(v_{14})} \not\in R_1(v_2)$ and $p_2^{(v_{14})} \in R_2(v_2)$.
Since $\{v_1,v_2,v_4,v_5,v_{14},v_{25}\}$ induces an $\overline{\mathrm{H}}$ and
$(v_1,v_2,v_4,v_5)$ is a consecutively tail-biting sequence of Type 3,
it follows from Lemma~\ref{lem:4triangles} and $p_2^{(v_{14})} \in R_2(v_2)$ that
\begin{equation}\label{eqn:six}
p_1^{(v_{25})} \in R_1(v_4).
\end{equation}
By (\ref{eqn:three}) and (\ref{eqn:six}),
since $A(v_4)$ and $A(v_{25})$ are homothetic,
we have
\begin{equation}\label{eqn:six-2}
p_2^{(v_{25})} \in R_2(v_4).
\end{equation}
Since $\{v_2,v_4,v_5,v_{25}\}$ induces a diamond and
$(v_2,v_4,v_5)$ is a consecutively tail-biting sequence of Type 3,
it follows from Lemma~\ref{lem:3triangles} that
$p_1^{(v_{25})} \in R_1(v_4)$ and $p_2^{(v_{25})} \notin R_2(v_4)$ hold or $p_1^{(v_{25})} \notin R_1(v_4)$ and $p_2^{(v_{25})} \in R_2(v_4)$ hold,
which is a contradiction to the fact that both (\ref{eqn:six}) and (\ref{eqn:six-2}) hold.
Thus
$p_1^{(v_{14})} \in R_1(v_2)$ and $p_2^{(v_{14})} \not\in R_2(v_2)$.

Since $A(v_3)$ and $A(v_{14})$ are homothetic,
we have
\begin{equation}\label{eqn:five-2}
p_1^{(v_{14})} \in R_1(v_3),
\end{equation}
contradicting (\ref{eqn:five}).

In the case where $p_1^{(v_{13})} \not\in R_1(v_2)$ and $p_2^{(v_{13})} \in R_2(v_2)$, we also reach a contradiction by applying a similar argument.

Hence, $\dim_{{\rm poc}}(G_n) > 3$ holds for any $n \geq r(5,5,5)$.
\end{proof}


\end{document}